\documentclass[12pt,reqno]{amsart}
\usepackage{amsmath, mathrsfs, dsfont, hyperref, amssymb}
\usepackage{amsthm}
\usepackage[top=1in, bottom=1in, left=1in, right=1in]{geometry}
\usepackage{xcolor}
\usepackage{pdfpages}
\usepackage{graphicx, wrapfig}
\usepackage{caption}
\usepackage{float}
\usepackage{adjustbox}

\theoremstyle{definition}
\newtheorem{theorem}{Theorem}[section]
\newtheorem{conjecture}[theorem]{Conjecture}
\newtheorem{corollary}[theorem]{Corollary}

\newtheorem{lemma}[theorem]{Lemma}
\newtheorem{question}[theorem]{Question}

\begin{document}

\title[End Behavior of Ramanujan's Taxicab Numbers]{End Behavior of Ramanujan's Taxicab Numbers}

\author{Brennan Benfield, Oliver Lippard, and Arindam Roy}
\address{Brennan Benfield: Department of Mathematics and Statistics, University of North Carolina at Charlotte, 9201 University City Blvd., Charlotte, NC 28223, USA}
\address{Oliver Lippard: Department of Mathematics and Statistics, University of North Carolina at Charlotte, 9201 University City Blvd., Charlotte, NC 28223, USA}
\address{Arindam Roy: Department of Mathematics and Statistics, University of North Carolina at Charlotte, 9201 University City Blvd., Charlotte, NC 28223, USA}

\thanks{The authors extend their sincerest gratitude to Michael De Vlieger for his expert assistance with Mathematica.}

\subjclass[2020]{11P05, 11D09, 11D25, 11P81, 11Y50}
\keywords{Taxicab Number, Waring's problem, quadratic and cubic Diophantine equations, partitions}

\begin{abstract}
Generalized taxicab numbers are the smallest positive integers that are the sum of exactly $j$, positive $k$\textsuperscript{th} powers in exactly $m$ distinct ways. This paper is considers for which values of $m$ does a smallest such integer exist as $j$ gets large. There appear to be only two possible outcomes, leading to curious results like there is no positive integer that can be expressed as the sum of exactly $10$ positive squares in exactly $3$ ways. This paper resolves a number of conjectures found in the OEIS by considering generalized Taxicab numbers in the setting of the theory of partitions. 
\end{abstract}
\maketitle

\section{Introduction}
In his biographical lectures titled \emph{Ramanujan} \cite{Hardy}, G. H. Hardy reflects on the famous exchange that occurred some years earlier:
\begin{quote}
    He could remember the idiosyncrasies of numbers in an almost uncanny way. It was Littlewood who said that every positive integer was one of Ramanujan’s personal friends. I remember going to see him once when he was lying ill in Putney. I had ridden in taxi-cab No. 1729, and remarked that the number seemed to me rather a dull one, and that I hoped that it was not an unfavourable omen. “No,” he replied, “it is a very interesting number; it is the smallest number expressible as a sum of two cubes in two different ways.”
\end{quote}
That must have startled Hardy; in fact, Ramanujan had previously studied Euler's Diophantine equation: $A^3+B^3=C^3+D^3$ and had written in his famous \textit{Lost Notebook} \cite{Bernt} some remarkable parametric solutions. 

 As Ramanujan and Hardy's conversation has been retold, $Ta(m)$ has come to be known as the $m$\textit{th} taxicab number, defined as the least integer which can be expressed as a sum of two positive cubes in $m$ distinct ways, up to order of summands. Trivially, $Ta(1)=2$ since $2=1^3+1^3$. Ramanujan made famous $Ta(2)=1729$, but it was known as early as 1657, mentioned in a correspondence between Wallis and Frenicle \cite{Frenicle}. 
 With computer assistance, 
 $Ta(3)$ was found in 1957 by Leech \cite{Leech}, followed by $Ta(4)$ in 1989 by Dardis \& Rosenstiel \cite{Dardis}, $Ta(5)$ in 1999 by Wilson \cite{Wilson}, and $Ta(6)$ in 2008 by Hollerbach \cite{Hollerbach}, who confirmed a conjecture made five years earlier by Calude, Calude \& Dinneen \cite{Claude}. No further taxicab numbers are known, although upper bounds were given in 2006 by Boyer \cite{Boyer} for $Ta(7)$ through $Ta(12)$. Despite knowing so few taxicab numbers, Fermat proved \cite{Boyer} that for all $m$, there exists a positive integer that can be expressed as the sum of two cubes in $m$ distinct ways. 
 \begin{table}[h!]
    \centering
    \begin{tabular}{lr}
        $2=1^3+1^3$ &$48988659276962496$ \\
        &$38787^3+365757^3=$\\
        $1729=12^3+1^3$ &$107839^3+362753^3=$ \\
        $\textcolor{white}{1729}=9^3+10^3$&$205292^3+342952^3=$   \\
        &$221424^3+336588^3=$   \\
        $87539319= 1673^3 + 4363^3$&$231518^3+331954^3=$   \\
        $\textcolor{white}{87539319}= 2283^3 + 4233^3$&   \\
        $\textcolor{white}{87539319}= 2553^3 + 4143^3$&$24153319581254312065344$   \\
        &$582162^3+28906206^3=$\\
        $6963472309248 = 24213^3 + 190833^3$&$3064173^3+28894803^3=$\\
        $\textcolor{white}{6963472309248}= 54363^3 + 189483^3$&$8519281^3+28657487^3=$   \\
        $\textcolor{white}{6963472309248}= 102003^3 + 180723^3$&$16218068^3+27093208^3=$   \\
        $\textcolor{white}{6963472309248} = 133223^3 + 166303^3$&$17492496^3+26590452^3=$ \\
        &$18289922^3+26224366^3=$
    \end{tabular}
    \caption{The six known taxicab numbers.}
\end{table}

Numbers that can be expressed as the sum of $j$ $k$\textsuperscript{th} positive powers in $m$ distinct ways have come to be known as \textit{generalized taxicab numbers}.

In 2019, Dinitz, Games, \& Roth \cite{Dinitz} denote by $T(k,j,n)$ the smallest number that can be expressed as the sum of $j$ positive $k$\textsuperscript{th} powers in \textit{at least} $m$ ways. They prove that for all $k$ and $j$ there always exists a positive integer that can be expressed as the sum of $j$ positive $k$\textsuperscript{th} powers in \textit{at least} $m$ ways.

Let $\text{Taxicab}(k,j,n)$ denote the smallest positive integer that can be expressed as the sum of $j$ positive $k$\textsuperscript{th} powers in \textit{exactly} $m$ ways. For example, $\text{Taxicab}(1,2,2)=4$, $\text{Taxicab}(2,2,2)=50$, $\text{Taxicab}(3,2,2)=1729$, and $\text{Taxicab}(4,2,2)=635318657$, but $\text{Taxicab}(5,2,2)$ is unknown \cite{Claude}. It could be that no such number exists - this paper gives evidence that for many values of $m$, $\text{Taxicab}(k,j,m)$ has no solution, and proves nonexistence for several particular cases.
\begin{theorem}\label{10squares3ways}
There is no positive integer that can be expressed as the sum of 10 or more positive squares in exactly 3 ways.
\end{theorem}
\noindent In other words, $\text{Taxicab}(2,10,3)=\emptyset$. It is not obvious for which values of $m$ there always exists a value for $\text{Taxicab}(k,j,m)$. This paper also establishes an upper bound $N^k(m,j)$ on how large $n$ needs to be to determine if $\text{Taxicab}(k,j,m)$ exists for $k=2$. 
\begin{theorem}\label{UpperBound}
    For $k=2$, define $N^2(m,j) =  (mj+j+14)^2$, then for each positive integer $j>6$ and $m>1$, if no solution to $\text{Taxicab}(2,j,m)$ exists for $n<N_m^j$, then no solution exists, else $\text{Taxicab}(2,j+1,m) = \text{Taxicab}(2,j,m)+1$ for all $n > N_m^j$.
\end{theorem}
With the techniques established in the proof of Theorem \ref{UpperBound}, this paper also resolves several conjectures found in the OEIS \cite{OEIS} regarding sequences of integers that cannot be written as the sum of exactly $j$ positive squares. 
\subsection*{Entry A080673 \cite{A080673}} Define the sequence where $A080673(n)$ as the ``largest number with exactly $n$ representations as a sum of five positive squares (or 0 if no number with exactly $n$ representations exists)". There is a comment from Johnson \cite{A080673} that $A080673(0)=33$ (i.e. $33$ is the largest positive integer with $0$ representations as a sum of $5$ positive squares) and that there is no number less than $10^6$ that can be expressed as the sum of five positive squares in exactly 188 ways (i.e. $A080673(188)=0$). Hasler asks if \textit{any} of the undefined $0$'s in the sequence have been proven. The fact that $A080673(0)=33$ follows immediately from Lemma \ref{Dubouis_2}. Using a similar technique as in the proof of Theorem \ref{10squares3ways}, it is possible to prove that in fact no number can be expressed as the sum of $5$ positive squares in exactly $188$ ways.
\begin{theorem}\label{Taxicab(2,5,188)}
    $\text{Taxicab}(2,5,188)$ does not exist.
\end{theorem}

\subsection*{Entry A295702 \cite{A295702}} Define the sequence $A295702(n)$ as the ``largest number with exactly n representations as a sum of $6$ positive squares (or 0 if no number with exactly n representations exists)". This entry contains the conjecture by Price: ``It appears that $A295702(36)$ does not exist."
\begin{theorem}\label{Taxicab(2,6,36)}
    $\text{Taxicab}(2,6,36)$ does not exist.
\end{theorem}

\subsection*{Entry A295795 \cite{A295795}}  
  Define the sequence $A295795(n)$ as the ``largest number with exactly n representations as sum of $7$ positive squares (or 0 if no number with exactly n representations exists)". This entry contains the conjecture by Price: ``It appears that $A295795(44)$ does not exist."
\begin{theorem}\label{Taxicab(2,7,44)}
    $\text{Taxicab}(2,7,44)$ does not exist.
\end{theorem}

It is natural to consider the end behavior of $\text{Taxicab}(k,j,m)+1$ for higher powers of $k$. After proving the main Theorems of this paper, numerical evidence is given that suggests a very similar result holds for $k>2$, with different particular instances of nonexistence.  
\begin{conjecture}\label{cubes_conj}
For all $m\geq2$ and for a fixed $k\geq2$, there exists a $J$ such that for all $j\geq J$, either there exists a positive integer that satisfies $\text{Taxicab}(k,j+1,m)=\text{Taxicab}(k,j,m)+1$, or $\text{Taxicab}(k,j,m)$ does not exist for all $j\geq J$.
\end{conjecture}
In particular, the data for cubes suggests a few particular cases. But without better theory, the computer search necessary to prove these instances is large and out of reach. 
\begin{conjecture}\label{47cubes41ways}
There is no positive integer that can be expressed as the sum of 47 positive cubes in exactly 41 ways.
\end{conjecture}
\begin{conjecture}\label{98squares90ways}
There is no positive integer that can be expressed as the sum of 98 or more positive cubes in exactly 90 ways.
\end{conjecture}

\section{Partitions}
Another way to frame Ramanujan's Taxicab numbers is within the setting of integer partitions. The partition function $p(n)$ enumerates the number of partitions of a positive integer $n$ where the partitions are positive integer sequences $\lambda=(\lambda_1,\lambda_2,...)$ with $\lambda_1\geq\lambda_2\geq~\dots>0$ and $\sum_{j\geq1}{\lambda_j}=n$. For example, $p(4)=5$ since the only ways to partition $4$ are $4$, $3+1$, $2+2$, $2+1+1$, and $1+1+1+1$. The sequence given by $p(n)$ has a generating function \cite{Abramowitz,Andrews} given by:
\[
\sum_n p(n)x^n = \prod_{i=1}^\infty \frac{1}{1-x^{i}} = \frac{1}{(1-x)(1-x^2)(1-x^3)\cdots}.
\]

It has become increasingly popular to consider \textit{restricted} partition functions, typically denoted $p_A(n)$ where $A$ is some constraint on $\lambda$. For instance, the integer $4$ could be partitioned using only prime numbers, using only the divisors of $4$, using only powers of $2$, using only odd numbers, etc:
\begin{alignat*}
4  & 4= 2+2  &\ \ \qquad  4 &  = 4            &\ \ \qquad  4  & = 2^2  &\ \ \qquad  4  & = 3+1\\
   &        &          & = 2+2 &           & = 2^1+2^1  &                       &=1+1+1+1\\
  &        &         & =1+1+1+1             &           & = 2^0+2^0+2^0+2^0  &         
\end{alignat*}

There are two restricted functions that are of particular interest regarding generalized taxicab numbers: the $k$\textsuperscript{th} power partition function and the partitions into exactly $j$ parts function. 

The $k$-th power partition function (see \cite{Benfield,[Gafni],[Vaughan],[Wright]}), denoted $p^k(n)$, counts the number of ways in which a positive integer $n$ can be expressed as the sum of positive $k$-th powers. Note that $p^k(n)=p(n)$ when $k=1$, and for $k=2$, $p^2(n)$ restricts $\lambda$ to positive squares. For example, $p^2(4)=2$ since $4=2^2=1^2+1^2+1^2+1^2$ and $p^k(4)=1$ for all $k\geq 3$. The generating function for the sequence given by $p^k(n)$: 
\[
\sum_n p^k(n)x^n = \prod_{i=1}^\infty \frac{1}{1-x^{i^k}} = \frac{1}{(1-x^k)(1-x^{2^k})(1-x^{3^k})\cdots}.
\]

The partitions into exactly $j$ parts function (see \cite{ahlgren,kim}), denoted $p(n,j)$ counts the number of ways that a positive integer $n$ can be written as a sum of exactly $j$ integers. For example, $p(4,2)=2$ because there are exactly two ways to write $4$ as a sum of $2$ integers: $4=3+1=2+2$. The sequence given by $p(n,j)$ has a generating function \cite{Andrews} given by:
\[
\sum_{n,j} p(n,j)x^n y^j =  \prod_{i=1}^j \frac{1}{1-yx^{i}} = \frac{1}{(1-yx)(1-yx^2)(1-yx^3)\cdots}.
\]

Synthesizing the combinatorial properties of these two restricted partition functions creates the \textit{$k$\textsuperscript{th} power partition into exactly $j$ parts} function, denoted $p^k(n,j)$. 
\begin{theorem}
    The generating function for $p^k(n,j)$ is given by
    \[\sum_{n=0}^\infty\sum_{j=0}^\infty p^k(n,j) x^n y^j = \prod_{i=1}^\infty \frac{1}{1-x^{i^k}y} = \frac{1}{(1-yx^k)(1-yx^{2^k})(1-yx^{3^k})\cdots}.\]
\end{theorem}

\begin{proof}
    The generating function for the unrestricted partition function counts all partitions of $n$ with weight $x^n.$ It is possible to build a generating function for $p^k(n,j)$ in two variables by counting partitions of $n$ into $j$ parts (in this case, restricted to $k$-th powers) with weight $x^ny^j.$
    
    The power of $x$ counts the sum of the partition, and the power of $y$ counts the number of parts. Thus, if $i^k$ appears in the partition, then there should be a term of $x^{i^k}y,$ because it adds $i^k$ to the sum and $1$ to the number of parts. Hence, the term contributed to the generating function is $\frac{1}{1-x^{i^k}y}.$ Any term that can be expressed as $i^k$ may be inserted into the partition, hence
    \[ \sum_{n=0}^\infty\sum_{j=0}^\infty p^k(n,j) x^n y^j = \prod_{i=1}^\infty \frac{1}{1-x^{i^k}y}.\]
\end{proof}

Given a positive integer $n$, there is a relation between this restricted partition function and generalized taxicab numbers:
\[\text{Taxicab}(k,j,m)=n
 \qquad  \qquad p^k(n,j)=m.
\]
$\text{Taxicab}(k,j,m)$ is simply the smallest $n$ satisfying $p^k(n,j)=m$. Theorem \ref{10squares3ways} and Conjecture \ref{98squares90ways} can be restated: there is no positive integer $n$ such that $p^2(n,10)=3$ or $p^3(n,90)=98$. In fact, the sequence $p^k(n,j)$ for a fixed $k$ is monotone for successive values of $n$ and $k$. 
\begin{theorem} \label{identity_1}
For all $n, j \ge 2$, 
\[p^k(n,j) \leq p^k(n+1, j+1)
\]
with equality whenever $n < 2^kj$.
\end{theorem}

\begin{proof}
    Suppose $n$ is the sum of $j$, positive $k$\textsuperscript{th} powers in exactly $m$ ways. The smallest possible representation of $n$ without a $1^2$ term is 
    \[
    n = \underbrace{2^k+2^k+\ldots+2^k}_{j\ \text{times}} = 2^kj.
    \]
    For all $n < 2^kj$, every representation of $n$ as a sum of $j$ positive $k$\textsuperscript{th} powers contains a $1^2$ term which can be subtracted off to obtain the identity
    \[p^k(n,j) = p^k(n-1, j-1).
    \]
    And if $n \geq 2^kj$, then adding $1^2$ to each representation of $n$ as a sum of $j$ positive $k$\textsuperscript{th} powers yields $m$ representations of $n+1$ as the sum of $j+1$, positive $k+1$\textsuperscript{th} powers. But there may be other representations of $n+1$ that contain no $1^2$ term. Hence, 
    \[p^k(n,j) \leq p^k(n+1, j+1).
    \]
\end{proof}

Further restricting the summands of $p^k(n,j)$, denote by $p^k(n,j,\mu)$ the number of partitions of $n$ into $j$ positive $k$\textsuperscript{th} powers with a maximum part $\mu$. For example, $p^2(50,2,25)=1$ since this function counts $50=5^2+5^2$ but not $50=1^2+7^2$. Note that $p^k(n,j)=p^k(n,j,n)$ and for $n \geq \mu$, the following identity holds.
\begin{theorem}
\label{identity_3}
For all positive integers $n,k,j,\mu$, 
\[ p^k(n,j,\mu) = \sum_{\substack{i = 1\\i^k \le \mu}} p^k(n-i^k, j-1, i^k).
\]
\end{theorem}

\begin{proof}
    The largest possible part of each partition given by $p^k(n,j,\mu)$ is determined by the sequence $i^k$ for $i =1, 2,\ldots,\lfloor\sqrt[k]{\mu}\rfloor$. Let $S_k$ be the set of partitions with largest part $i^k.$ Then, the term $i^k$ can be removed from any partition in $S_k$ to obtain a partition of $n-i^k$ into $j-1$ parts. Since $i^k$ is the largest part in the partition before its removal, no parts in the resulting partition can be greater than $i^k.$ Thus, $|S_k| = p^k(n-i^k, j-1, i^k).$ Summing the number of partitions for each possible $k$ yields the identity.
\end{proof}

\begin{corollary}
    The generating function for $p^k(n,j,\mu)$ is given by \[ \sum_{\substack{i = 1\\i^k \le \mu}} p^k(n-i^k, j-1, i^k)x^ny^j = \prod_{\substack{i = 1\\i^k \le \mu}} \frac{1}{1-x^{i^k}y}.
\]
\end{corollary}
\begin{proof}
    This follows naturally from Theorem 2.2, with the generating function truncated at $\mu$ because no larger parts are allowed.
\end{proof}
Starting the product for the generating function of $p^k(n,j)$ (or $p^k(n,j,\mu)$) at $i = 0$ instead of $i = 1$ generates the sequence of partitions of $n$ into $j$ \textit{non-negative} $k$\textsuperscript{th} powers (with maximum part $\mu$). In other words, starting at $i = 0$ allows $0^k$ as a summand. This is also equal to the sequence of partitions of $n$ into $\textit{at most}$ $j$ positive $k$-th powers.

\section{Positive Squares}
Consider first when $k=2$, for which values of $m$ does there exist a positive integer that can be written as the sum of $j$ positive squares in exactly $m$ ways? For a fixed $m$, as $j$ increases one of two things will happen: increasing the value from $j$ to $j+1$ is accomplished by adding $1^2$ to each representation of $\text{Taxicab}(2,j,m)$, or no such number exists for all larger values of $j$; these are the only two possible cases.

\subsection*{Case 1}
Consider the sequence given by $\text{Taxicab}(2,j,2)$, the smallest number that can be expressed as the sum of $j$ positive squares in exactly 2 ways as $j$ increases. The smallest integer that can be expressed as the sum of two squares in exactly two ways is $50=1^2+7^2=5^2+5^2=\text{Taxicab}(2,2,2)$. The sequence begins:
\begin{table}[h!]
\begin{adjustbox}{width=1\textwidth}
\begin{tabular}{c|c|c|c|c|c|c|c|c|c|c}
&$j=1$&$j=2$&$j=3$&$j=4$&$j=5$&$j=6$&$j=7$&$j=8$&$j=9$&$j=10$\\
\hline
$\text{Taxicab}(2,j,2)$&2&50&27&31&20&21&22&23&24&25
\end{tabular}
\end{adjustbox}
\end{table}\\
Notice the pattern emerging at $j=5$: $\text{Taxicab}(2,j+1,2)$ is obtained by adding $1^2$ to every representation of $\text{Taxicab}(2,j,2)$. 
\begin{align}
\nonumber 20&=2^2+2^2+2^2+2^2+2^2\\
\nonumber &=4^2+1^1+1^2+1^2+1^2\\
\nonumber \\
\nonumber 21&=2^2+2^2+2^2+2^2+2^2\mathbf{\textcolor{blue}{+1^2}}\\
\nonumber &=4^2+1^1+1^2+1^2+1^2\mathbf{\textcolor{blue}{+1^2}}\\
\nonumber \\
\nonumber 22&=2^2+2^2+2^2+2^2+2^2+1^2\mathbf{\textcolor{blue}{+1^2}}\\
\nonumber &=4^2+1^1+1^2+1^2+1^2+1^2\mathbf{\textcolor{blue}{+1^2}}\\
\nonumber \vdots
\end{align}
This is one possible end behavior for a fixed $m$ as $j$ increases. 
\subsection*{Case 2}
The other possible end behavior for a fixed $m$ is that there exists a $J$ such that for all $j\geq J$, there does not exist a positive integer that satisfies $\text{Taxicab}(2,j,m)$. The smallest example occurs when $m=3$ and $j=10$; there does not exist an integer that can be expressed as the sum of $10$ or more squares in exactly $3$ ways. Consider the first few cases for $\text{Taxicab}(2,j,3)$.
\begin{table}[h!]
\begin{adjustbox}{width=1\textwidth}
\centering
\begin{tabular}{c|c|c|c|c|c|c|c|c|c|c}
&$j=1$&$j=2$&$j=3$&$j=4$&$j=5$&$j=6$&$j=7$&$j=8$&$j=9$&$j=10$\\
\hline
$\text{Taxicab}(2,j,3)$&3&325&54&28&29&30&31&35&49&${\emptyset}$
\end{tabular}
\end{adjustbox}
\end{table}
\\
But why not simply add $1^2$ to each representation of $49$ and obtain $T(2,10,3)=50$? In fact, $50$ has \textit{four} representations as the sum of $10$ squares:
\begin{align}
\nonumber \text{Taxicab}(2,9,3)=49&=5^2+3^2+3^2+1^2 +1^2+1^2+1^2+1^2+1^2\\
\nonumber &=3^2+3^2+3^2+3^2+3^2+1^2+1^2+1^2+1^2\\
\nonumber &=4^2+3^2+3^2+2^2+2^2+2^2+1^2+1^2+1^2\\
\nonumber \\
\nonumber 50&=5^2+3^2+3^2+1^2 +1^2+1^2+1^2+1^2+1^2\mathbf{\color{blue}{+1^2}}\\
\nonumber &=3^2+3^2+3^2+3^2+3^2+1^2+1^2+1^2+1^2\mathbf{\color{blue}{+1^2}}\\
\nonumber &=4^2+3^2+3^2+2^2+2^2+2^2+1^2+1^2+1^2\mathbf{\color{blue}{+1^2}}\\
\nonumber &\mathbf{\color{red}{\ =3^2+3^2+2^2+2^2+2^2+2^2+2^2+2^2+2^2+2^2}}
\end{align}
In other words, adding $1^2$ to each representation of $\text{Taxicab}(2,9,3)$ does not yield $\text{Taxicab}(2,10,3)$. This is the only other possible end behavior of $\text{Taxicab}(2,j,m)$ for a fixed $m$ as $j$ gets large.

\section{Numerical Data for $\text{Taxicab}(2,j,m)$}
Figure \ref{bitmap} is a bit map of $\text{Taxicab}(2,j,m)$ for $2 \leq j \leq 75$ running from left to right along the horizontal axis and $1 \leq m \leq 500$ running top to bottom along the vertical axis. This map shows a white pixel if $\text{Taxicab}(2,j,m)$ exists and a black pixel if not. 
\begin{figure}[h!]
\includegraphics[scale=.45]{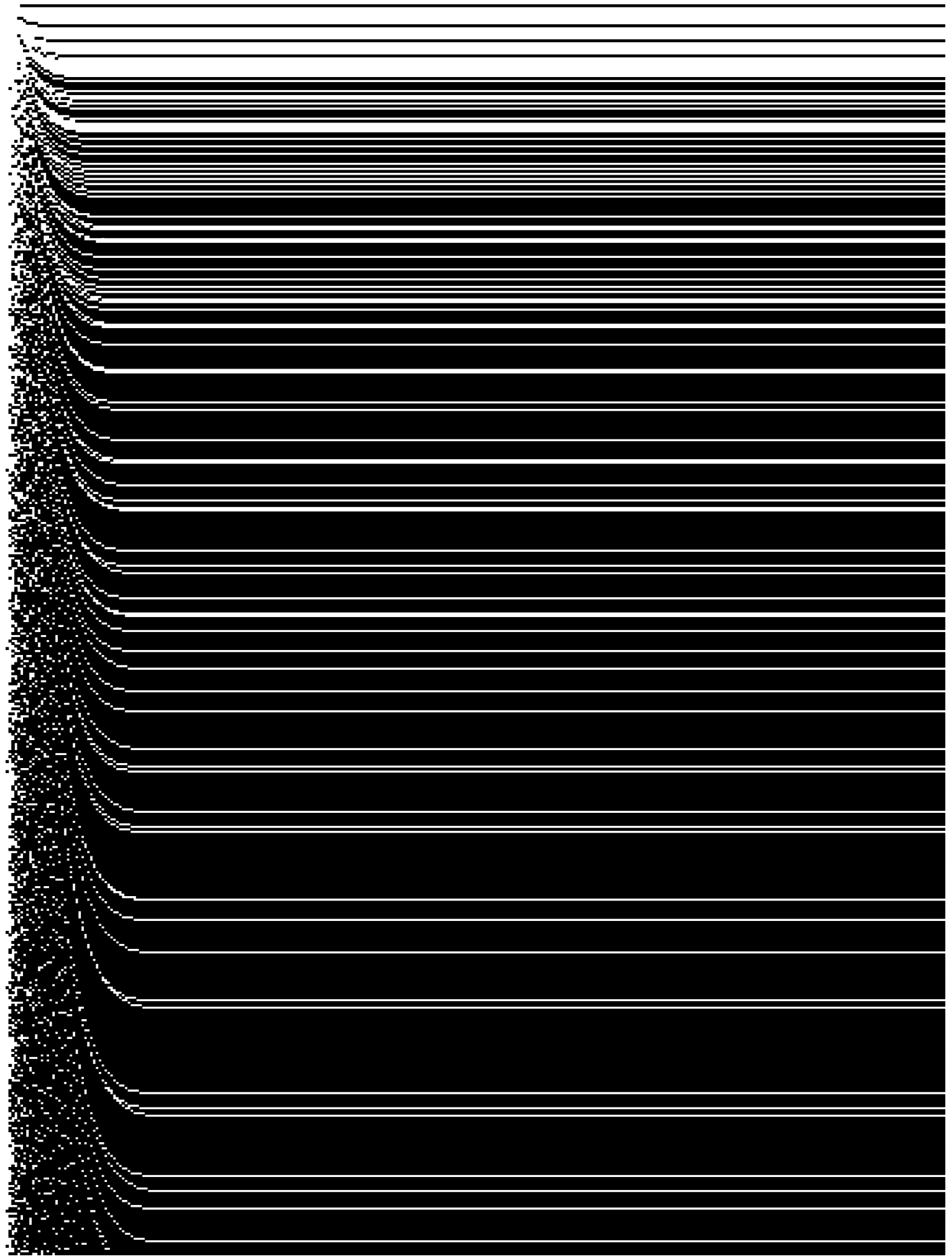}
\caption{$\text{Taxicab}(2,j,m)$ for $j\leq75$ and $m\leq500$.}
\label{bitmap}
\end{figure}

The end behavior of $\text{Taxicab}(2,j,m)$ splits the positive integers into two sets: numbers for which there always exists a solution for $\text{Taxicab}(2,j,m)$ for sufficiently large $j$, and the complement, numbers for which there does not exist a solution for $\text{Taxicab}(2,j,m)$ for sufficiently large $j$. Denote by $\{m_i^k\}$ the sequence of $m$ such that $\text{Taxicab}(k,j,m)$ has a solution for a fixed $k$ and for sufficiently large $j$, and denote by $\{\mathbb{N}\setminus m_i^k\}$ the sequence of $m$ such that no solution to $\text{Taxicab}(k,j,m)$ exists for sufficiently large $j$. Then $m_i^2$ and its complement sequence $\mathbb{N}\setminus m_i^2$ begin:
\begin{align*}  m_i^2&=1,2,4,5,6,7,8,9,10,12,16,14,15,16,18,19,20,21,22,24,\ldots\\
\mathbb{N}\setminus m_i^2&=3,11,17,23,32,34,35,36,38,41,43,45,46,47,49,54,55,57,\ldots
\end{align*}
\begin{conjecture}
    The sequences $\{m_i^2\}$ and $\{\mathbb{N}\setminus m_i^2\}$ each contain infinitely many elements.
\end{conjecture}
\begin{question}
    For the sequence $\{m_i^2\}$, is it possible to find an explicit or an asymptotic formula?
\end{question}
\noindent It appears that $\{m_i^2\}$ is growing at roughly exponential rate, indicating that values of $m$ for which $\text{Taxicab}(2,j,m)$ exists become increasingly rare as $m$ gets large. In other words, it appears that `most' values of $m$ do not allow a solution for $\text{Taxicab}(2,j,m)$. Utilizing Mathematica, a reasonable approximation for $\{m_i^2\}$ appears to be $c(x)\approx17.5603\exp(-0.0327825x)$. Figure \ref{m_i_squares} maps $\{m_i^2\}$ and $c(x)$ for $m \leq 500$.
\begin{figure}[h!]
\includegraphics[scale=.4]{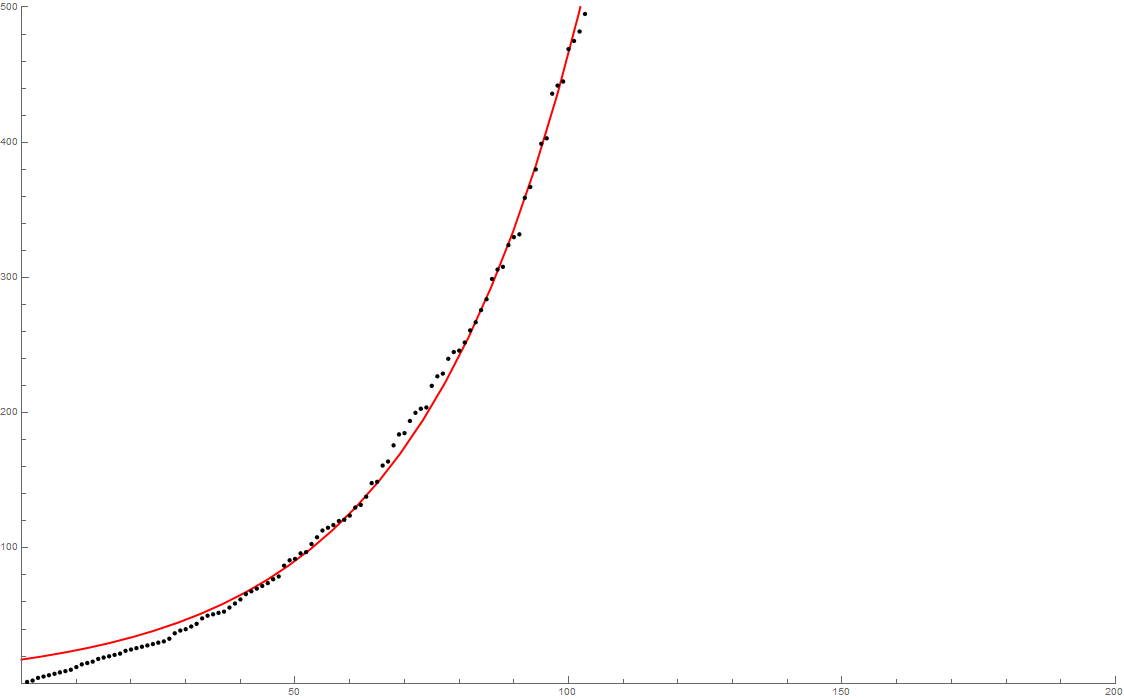}
\caption{$c(x)$ and $\{m_i^2\}$ for $m\leq500$.}
\label{m_i_squares}
\end{figure}

\section{Proof of Theorems \ref{10squares3ways} and \ref{UpperBound}}

To prove that there is no positive integer $n$ that can be written as the sum of 10 squares in exactly 3 ways, consider one of the 3 representations of $n$:
\begin{align*}    n&=\underbrace{x_1}_{a_1}+\underbrace{x_2+x_3+x_4+x_5+x_6+x_7+x_8+x_9+x_{10}}_{b_1}
\end{align*}

\noindent Let $n=a+b$ where $a = x_1^2$ and $b = n-a = x_2^2 + x_3^2 + x_4^2 + x_5^2 + x_6^2 + x_7^2 + x_8^2 + x_9^2 + x_{10}^2$. Then, whenever there are $3\binom{10}{1}+1 = 31$ ways to choose $a$ such that $a$ is a perfect square and $b$ is the sum of $9$ positive squares, there must be at least $4$ ways to express $n$ as the sum of $10$ positive squares. In 1911, Dubouis \cite{Dubouis} (and independently in 1933, Pall \cite{Pall}) determined exactly which positive integers are not the sum of exactly $9$ positive squares. Grosswald \cite{Grosswald} summarizes the result elegantly: let $\mathbf{B}=\{1, 2, 4, 5, 7, 10, 13\}$.
\begin{lemma}[Dubouis \& Pall]\label{Dubouis_2}
    Every positive integer $n$ can be expressed as the sum of $j \geq 6$ positive squares provided $n>j-1$ and $n \neq j+b$ for $\beta \in \mathbf{B}$. For $j=5$, the result holds for $\mathbf{B} \cup \{28\}$.
\end{lemma}
Lemma \ref{Dubouis_2} enumerates which positive integers cannot be expressed as the sum of $9$ squares: $1, 2, 3, 4, 5, 6, 7, 8$, and  $10, 11, 13, 14, 16, 19, 22$. Let $a \geq 961$; there are at least $\lfloor\sqrt{961}\rfloor = 31$ positive squares less than $a$. For all $n \geq 961 + 22 +1 = 984$, note that $n = a+b$ where $a$ is a perfect square and $b$ is a sum of $9$ squares, and there are at least $4$ ways to express $n$ as the sum of $10$ positive squares. A computer search verifies that no $n < 961$ can be expressed as the sum of $10$ positive squares in exactly $3$ ways.

In fact, $p^2(n,10) \geq 4$ for all $n \geq 48$. Note that $1^2$ can be added $c$ times to each representation of $n$, hence $p^2(n,j) \leq p^2(n+c,j+c)$ for all $j \geq 10$ and $n \geq 48$. In particular, for all $n$ and $j$, $p^2(n,j) \leq p^2(n+1,j+1)$. Similarly, if every representation of $n$ as the sum of $j$ squares in exactly $m$ ways has $1^2$ as a term, then $p^2(n,j)=p^2(n-1,j-1)$ by removing the $1^2$ term from each representation. The minimum that a representation with no $1^2$ term can be is $4j$ since the smallest such representation is simply $2^2$, added exactly $j$ times: 
\[
n = \underbrace{2^2+2^2+2^2+\ldots+2^2}_{j\ \text{times}}=4j
\]
Thus, for $n < 4j$, the identity $p^2(n,j) = p^2(n-1,j-1)$ holds. The same process outlined above for can be repeated for $j=11$ and $j=12$. A computer search up to $(11\cdot3+11+14)^2=3364$ ensures that $\text{Taxicab}(2,11,3)$ does not exist and a search up to $(12\cdot3+12+14)^2=3844$ ensures that $\text{Taxicab}(2,12,3)$ does not exist. In other words, there is no $n$ such that $p^2(n,11)=3$ nor $p^2(n,12)=3$. Let $j>12$ and suppose $n<38+j<4j$, then by Theorem \ref{identity_1}, 
\[
p^2(n,j)=p^2(n-1,j-1)=\ldots=p^2(n-j+12,12)\neq3.
\]
Otherwise, if $n \geq 38+j$, then $n-j+12 \geq 50$, and 
\[
p^2(n,j) \geq p^2(n-1,j-1) \geq \ldots \geq p^2(n-j+12,12) \geq 4. 
\]
This completes the proof of Theorem \ref{10squares3ways}. \qed

Proof of Theorem \ref{10squares3ways} relied on knowing which numbers are not a sum of $9$ (or, in general, $j-1)$ squares. By Lemma \ref{Dubouis_2}, the largest number that cannot be written as a sum of $j-1$ squares is $j-1+13$. Then,  there are $\binom{j}{1}=j$ ways to decompose a sum of $j$ squares into a square and a sum of $j-1$ squares. If there are at least $mj+1$ representations of $n$ as a perfect square plus a number that is the sum of $j-1$ squares, then there are at least $m+1$ representations of $n$ as a sum of $j$ squares by the Pigeonhole principle. This gives an upper bound on how high a computer needs to check to determine for a given $m$ if $\text{Taxicab}(2,j,m)$ exists for $j \ge 7$. Denote this upper bound by $N^2(m,j)$:
\[
N^2(m,j) = (mj+1+(j+13))^2 = (mj+j+14)^2
\]

\noindent In other words, to determine if there exists an $n$ that is the sum of $j$ squares in exactly $m$ ways, a computer need only search up to $N_m^j$. To complete the proof, consider the result of Dinitz, Games, \& Roth.
\begin{lemma}[Dinitz, Games, \& Roth]\label{lemma_D,G,R}
    For all $j,m \geq 1$, there exists a smallest positive integer $s_0$ such that $T(s_0+\lambda,j,m) = T(s_0,j,m)+\lambda$ for every $\lambda \geq 0$. 
\end{lemma}
By Lemma \ref{lemma_D,G,R}, for $n > N_j^k$, if $\text{Taxicab}(2,j,m)$ exists, then $\text{Taxicab}(2,j+1,m) = \text{Taxicab}(2,j,m)+1$; else, no such integer exists. This completes the proof of Theorem \ref{UpperBound}. \qed

A reliable bound is obtained from Theorem \ref{UpperBound}, but a tighter bound exists. Visible in Figure 1 is a curved boundary separating an unorganized region of small values of $j$ and a regularly organized region of large values of $j$. Theorem \ref{UpperBound} holds for all $n > N_j^k$, but frequently it also holds for many smaller values of $n$. Let $J^2(m)$ map the smallest $(m,J)$ such that Theorem \ref{UpperBound} holds for all $j \geq J$. Utilizing Mathematica, a close fit for $J^2(m)$ appears to be $g_2(x)=45.06\sqrt[8]{x}-44.7873$. Figure \ref{boundary_squares} plots $J(m)$ and $b(x)$ for $j \leq 75$ and $m \leq 500$.

\begin{figure}[h!]
\includegraphics[scale=.25]{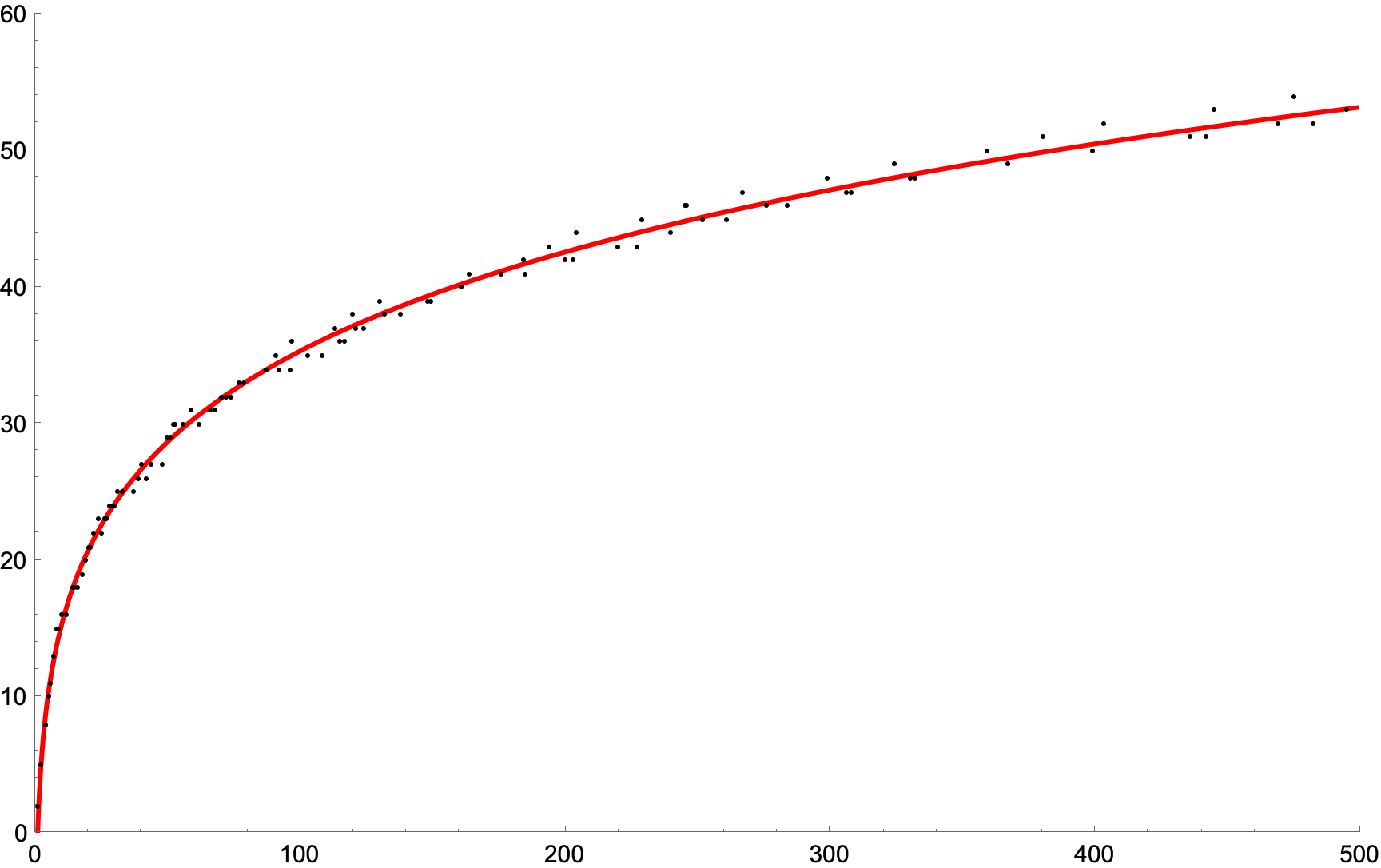}
\caption{$J^2(m)$ (black) and $g_2(x)$ (red)}
\label{boundary_squares}
\end{figure}

\section{Proofs of Conjectures from the OEIS}

\subsection{Proof of Theorem \ref{Taxicab(2,5,188)}}
Lagrange \cite{Lagrange} famously proved in 1770 that every number can be expressed as the sum of four squares, but allowed $0$ as a term in the representation - not every integer has a representation as the sum of four \textit{positive} squares. The sequence of positive integers that cannot be expressed as the sum of four squares begins:
\[
A000534(n) =  1, 2, 3, 5, 6, 8, 9, 11, 14, 17, 24, 29, 32, 41, 56, 96, 128, 224, 384, \ldots
\]
The complete list of such positive integers was conjectured in 1638 by Descartes \cite{Descartes} and proven in 1911 by Dubouis \cite{Dubouis}. Recall that $\mathbf{B}=\{1, 2, 4, 5, 7, 10, 13\}$.
\begin{lemma}[Dubouis]\label{Dubouis_1}
     Every positive integer $n$ can be expressed as the sum of $4$ positive squares provided $n > 3$ and $n \neq 4+b$ for $\beta \in \mathbf{B}\cup\{25,37\}$ and $n \neq 2*4^\alpha, 6*4^\alpha, 14*4^\alpha$ for $\alpha=0,1,2,\ldots$.
\end{lemma}
\begin{proof}
  Consider an arbitrary representation of a positive integer $n$ as the sum of $5$ positive squares. Let $n=a+b$ where $a=x_1^2$ and $b=x_2^2+x_3^2+x_4^2+x_5^2$. By Lemma \ref{Dubouis_1}, the only positive integers greater than $41$ that cannot be expressed as a sum of exactly four squares are $2\cdot4^\alpha, 6\cdot4^\alpha,$ and $14\cdot4^\alpha$ for $\alpha = 1, 2, \ldots$ For all $n \ge 921681$, there are at most $21$ integers between $n - 919681$ and $n$ that cannot be expressed as a sum of four squares.  
  
  Let $\gamma(n)$ be the number of positive integers between $n$ and $n-919681$ such that $n$ cannot be expressed as a sum of exactly four squares. Since $n \ge 921681, n - 919681 \ge 2000.$ Then, the smallest value of $\alpha$ that yields an integer that \textit{cannot} be expressed as a sum of four squares is $\lceil \log_4\left(\frac{n-919681}{2}\right)\rceil$, and the largest value of $\alpha$ that \textit{can} be expressed as a sum of four squares is $\lfloor\log_4\left(\frac{n}{14}\right)\rfloor.$ Note that for each value of $\alpha$, there are at most $3$ positive integers that cannot be expressed as the sum of four positive squares. It follows that $\gamma(n) \le 3\left(\log_4\left(\frac{n}{2}\right)-\log_4\left(\frac{n-919681}{14}\right)\right),$ which is a decreasing function. Evaluation at $n=921681$ shows $\gamma(n) \le 18$ for all $n \ge 921681.$ Therefore, there are at least $\lfloor \sqrt{919681}-18 \rfloor = 941$ combinations of a square and a sum of four squares that sum to $n$. At most $\binom{5}{4}=5$ of these sums give the same combination of $5$ squares. Hence, for all $n > 921681$, there are at least $\lceil \frac{941}{5} \rceil = 189$ ways to represent $n$ as a sum of $5$ positive squares. Johnson \cite{A080673} has already checked that $\text{Taxicab}(2,5,188)$ does not exist for $n < 10^6$.
  \end{proof}
Hasler \cite{A080673} asks if any of the other listed elements in the sequence where $A080673(n)=0$ are proven. This paper establishes the nonexistence of $\text{Taxicab}(2,5,m)$ for $m \leq 500$, and proves some other values of $n$ such that $A080673(n)=0$.
\begin{theorem}
  $\text{Taxicab}(2,5,m)$ does not exist for $m = 259, 304, 308, 372, 394, 483, 497$.
\end{theorem}

\subsection{Proof of Theorem \ref{Taxicab(2,6,36)}}
\begin{proof}
     A computer search verifies that no $n \leq N^2(36,6)=(6\cdot36+6+14)^2=55696$ can be expressed as the sum of $6$ positive squares in exactly $36$ ways. By Theorem \ref{UpperBound}, $\text{Taxicab}(2,6,36)$ does not exist.
  \end{proof}
  \noindent Other values of $n$ such that $A295702(n)=0$ are established for $m<500$.

  \begin{theorem}
  $\text{Taxicab}(2,6,m)$ does not exist for\\
  \\
  $m =$\ 70, 82, 99, 116, 124, 126, 139, 140, 147, 162, 164, 165, 171, 182, 190, 193, 197, 205, 206, 207, 212, 214, 218, 222, 227, 231, 240, 243, 256, 273, 277, 280, 285, 287, 288, 291, 292, 300, 302, 322, 330, 334, 339, 344, 346, 347, 353, 356, 360, 364, 372, 379, 380, 383, 385, 392, 395, 396, 398, 402, 405, 407, 408, 410, 411, 412, 413, 425, 431, 432, 435, 436, 437, 439, 442, 443, 446, 448, 450, 451, 457, 466, 472, 474, 476, 482, 483, 485, 488, 491, 492, 493, 500.
\end{theorem}

\subsection{Proof of Theorem \ref{Taxicab(2,7,44)}}
\begin{proof}
     A computer search verifies that no $n \leq N^2(44,7)=(7\cdot44+7+14)^2=107584$ can be expressed as the sum of $7$ positive squares in exactly $44$ ways. By Theorem \ref{UpperBound}, $\text{Taxicab}(2,7,44)$ does not exist.
  \end{proof}
  \noindent Other values of $n$ such that $A295795(n)=0$ are established for $m<500$.

  \begin{theorem}
  $\text{Taxicab}(2,7,m)$ does not exist for\\
  \\
  $n =$\ 47, 59, 63, 67, 74, 81, 90, 97, 105, 106, 108, 110, 111, 112, 119, 120, 122, 125, 126, 131, 132, 140, 142, 143, 148, 151, 153, 158, 163, 166, 168, 169, 171, 174, 175, 176, 182, 189, 190, 192, 195, 196, 198, 199, 204, 207, 208, 211, 213, 215, 217, 222, 224, 225, 226, 237, 238, 239, 240, 242, 244, 246, 247, 250, 253, 255, 257, 260, 262, 266, 267, 269, 271, 272, 273, 276, 278, 279, 280, 283, 285, 289, 292, 293, 296, 297, 298, 299, 300, 301, 303, 304, 307, 314, 315, 319, 321, 322, 325, 326, 327, 329, 330, 332, 334, 338, 339, 341, 342, 343, 344, 347, 349, 353, 357, 358, 360, 361, 362, 363, 364, 366, 370, 371, 372, 376, 378, 384, 385, 386, 388, 389, 390, 392, 393, 394, 396, 399, 401, 402, 403, 404, 405, 406, 408, 411, 413, 419, 421, 427, 428, 429, 432, 433, 434, 435, 437, 438, 439, 440, 441, 442, 444, 445, 448, 449, 451, 452, 453, 455, 460, 465, 466, 467, 469, 470, 471, 472, 474, 476, 477, 479, 480, 482, 483, 485, 489, 490, 493, 495, 496, 497, 500.
\end{theorem}

\subsection*{Entry A025416 \cite{A025416}} There is a conjecture in the OEIS that illuminates the behavior of $\text{Taxicab}(2,j,m)$ for a fixed $j$. Define the sequence $A025416(n)$ as the sequence given by $\text{Taxicab}(2,4,m)$ for $m=1,2,\ldots$ Perry \cite{A025416} conjectures that the sequence $A025416(n)$ never becomes monotone. 
\begin{figure}[h!]
    \centering    \includegraphics[scale=.4]{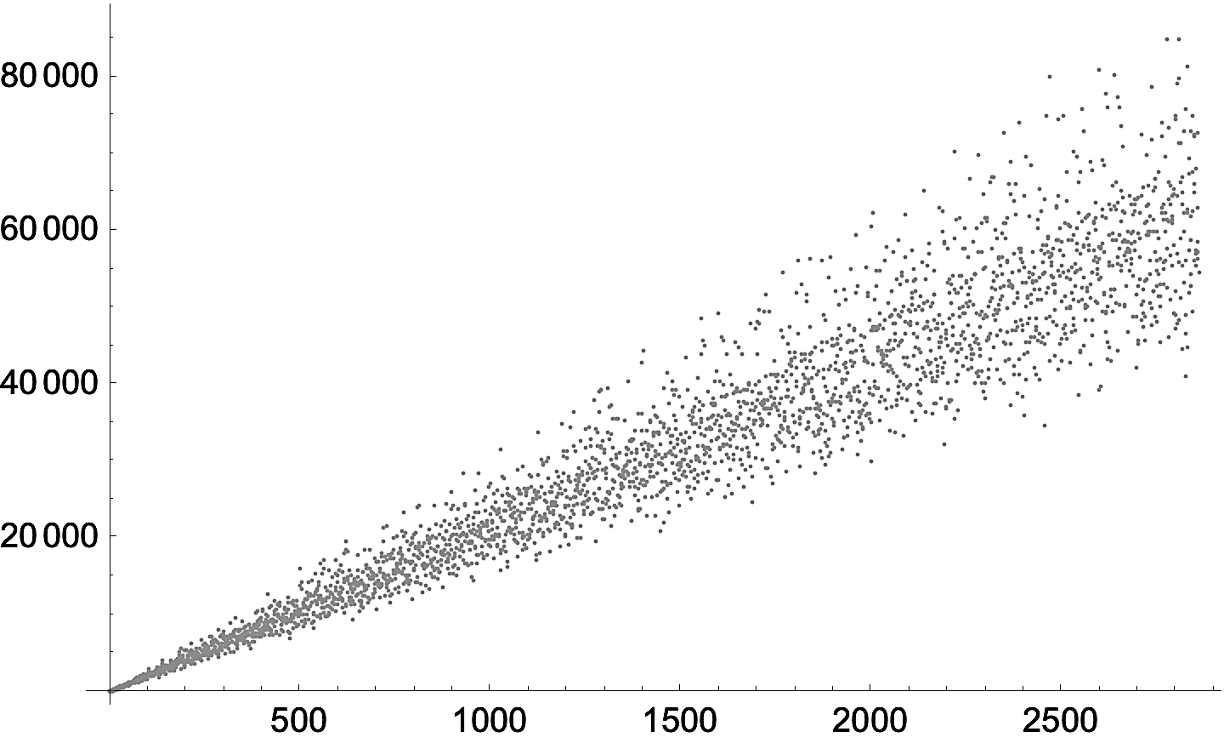}
    \caption{$\text{Taxicab}(2,4,m)$ for $m \leq 2861$.}
    \label{fig:enter-label}
\end{figure}
\[
\text{Taxicab}(2,4,m): 4, 31, 28, 52, 82, 90, 135, 130, 162, 198, 202, 252, 234, 210, \ldots
\]
for $m=1,2,\ldots$ Perry's conjecture appears to hold for all $j$ and is generalized in the following conjecture. 
\begin{conjecture}[Perry]\label{Perry}
    Given a fixed $j>2$, as $m$ increases, the sequence given by $\text{Taxicab}(2,j,m)$ never becomes monotone.
\end{conjecture}
There is good evidence to believe Conjecture \ref{Perry}. Note that these sequences are essentially the columns of Figure \ref{bitmap}, and most values of $m$ and $j$ lie in the highly regular region on the right hand side. On the OEIS, values are $0$ whenever $\text{Taxicab}(2,j,m)$ does not exist, and some positive integer whenever $\text{Taxicab}(2,j,m)$ does exist. By that alone, it appears that Perry's conjecture holds, supposing there are infinitely many elements of both $m_i$ and $\mathbb{N}\setminus m_i$. But Perry's conjecture made for $m=4$ contains no values of $0$, that is, $\text{Taxicab}(2,4,m)$ always exists (at least as far as OEIS A025416 has entries). Define $m_{i,j}^2$ as the sequence of $m$ such that $\text{Taxicab}(2,j,m)$ exists. Then, Perry's conjecture may be generalized by the following:
\begin{conjecture}
    For each $m$ and $j$, the sequence $m_{i,j}^2$ never becomes monotone.
\end{conjecture}

\begin{conjecture}
    For $j<4$ and for all $m$, there always exists a positive integer $n$ such that $\text{Taxicab}(2,j,m) = n$.
\end{conjecture}

\section{Higher Positive Powers}
For which values of $m$ does there exist a positive integer that can be written as the sum of $j$ positive $k$\textsuperscript{th} powers in exactly $m$ ways? Analogous to the end behavior in the setting of positive squares, it appears that for a fixed $m$, as $j$ increases, one of two things occurs for $\text{Taxicab}(k,j,m)$.

Conjecture \ref{cubes_conj} is supported by numerical evidence for $k=3$ and $4$. Figure \ref{bitmap_cubes} (Left), is a bitmap of $\text{Taxicab}(3,j,m)$, where conjectured values of $j$ for $7 \leq j \leq 200$ run along the horizontal axis and values of $m$ for $1 \leq m \leq 500$ run down the vertical axis. Again, a white pixel indicates that a solution to $\text{Taxicab}(3,j,m)$ exists and a black pixel indicates that no such solution exists. 
\begin{figure}[h!]
    \centering
    \includegraphics[scale=.5]{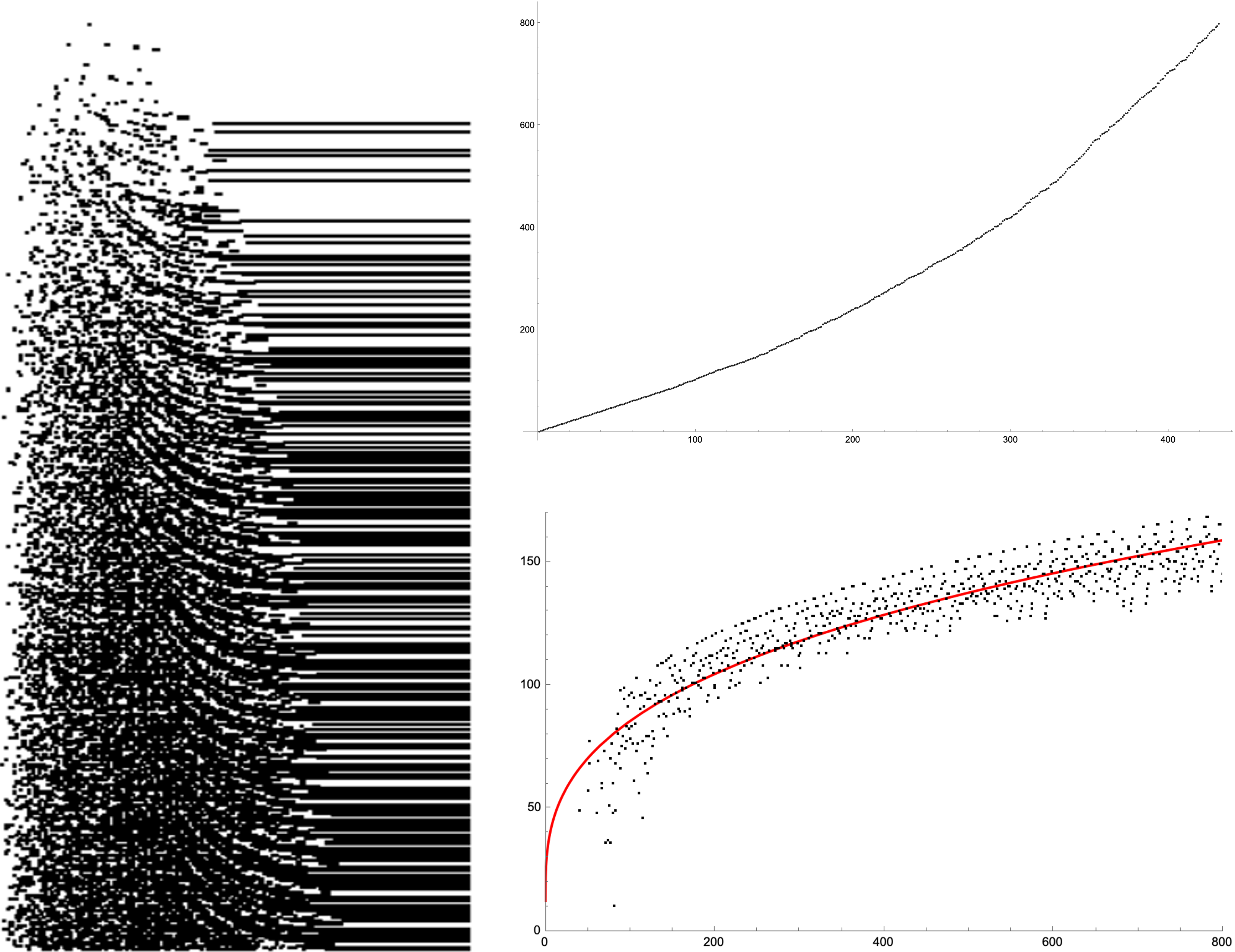}
    \caption{(Left) $\text{Taxicab}(3,j,m)$, (Right Top) $m_i^3$, and (Right Bottom) $J^3(m)$ and $g_3(m)$ for $m \leq 500$ and $7 \leq j \leq 200$}
    \label{bitmap_cubes}
\end{figure}
A familiar pattern emerges in Figure \ref{bitmap_cubes}. There is a chaotic region on the left corresponding to small values of $j$ and a highly regular region on the right corresponding to large values of $j$. These two regions of the graph are again separated by a boundary curve with a familiar shape. Let $J^3(m)$ denote the smallest pair $(m,J)$ such that Conjecture \ref{cubes_conj} holds for all $j \geq J$. Figure \ref{bitmap_cubes} (Right Top) graphs $J^3(m)$ for $m \leq 500$. Utilizing Mathematica, a close fit for $J^3(m)$ appears to be $g_3(x) = 15.8278\sqrt[3]{x}+12.0661$

Values of $m$ for which $\text{Taxicab}(3,j,m)$ exists creates the sequence $\{m_i^3\}$ and its complement sequence $\{\mathbb{N} \setminus m_i^3\}$. Figure \ref{bitmap_cubes} (Right Bottom) plots $J^3(m)$ for $m \leq 500$.

\begin{align*}  
m_i^3&=1,2,3\ldots89,91, 92, 93, 95, 96, 97, 98, 99, 100, 101, 102, 104, 105,\ldots\\
\mathbb{N} \setminus m_i^3&=90, 94, 103, 106, 113, 118, 138, 146, 149, 156, 157, 160, 164, 165,\ldots
\end{align*}

For which values of $m$ does there exist a positive integer that can be written as the sum of exactly $j$ positive fourth powers in exactly $m$ ways? This was G. H. Hardy's question precisely when, after hearing Ramanujan's reply regarding $1729$, Hardy inquired: \cite{Hardy}
\begin{quote}
    I asked him, naturally, whether he could tell me the solution of the corresponding problem for fourth powers; and he replied, after a moment’s thought, that he knew no obvious example, and supposed that the first such number must be very large.
\end{quote}
The answer to Hardy's question is large indeed - the discovery that 
\[
635318657=59^4+158^4=133^4+134^4
\]
is credited to Euler \cite{Euler} but it wasn't until 1957 that Leech \cite{Leech} proved $635318657$ is in fact $\text{Taxicab}(4,2,2)$. The existence of positive integers that are the sum of $2$ positive $k$\textsuperscript{th} powers in exactly $2$ ways is still open. In terms of Diophantine equations, it is unknown \cite{Ekl} if the equation $A^k+B^k=C^k+D^k$ has solutions for $k > 4$. 

Conjecture \ref{cubes_conj} is further supported by Figure \ref{bitmap_fourths} (Left), a bitmap of $\text{Taxicab}(4,j,m)$, where values of $j$ for $j \leq 500$ run along the horizontal axis and values of $m$ for $1 \leq m \leq 500$ run down the vertical axis. Again, a white pixel indicates that a solution to $\text{Taxicab}(4,j,m)$ exists and a black pixel indicates that no such solution exists. 
\begin{figure}[h!]
    \centering
    \includegraphics[scale=.45]{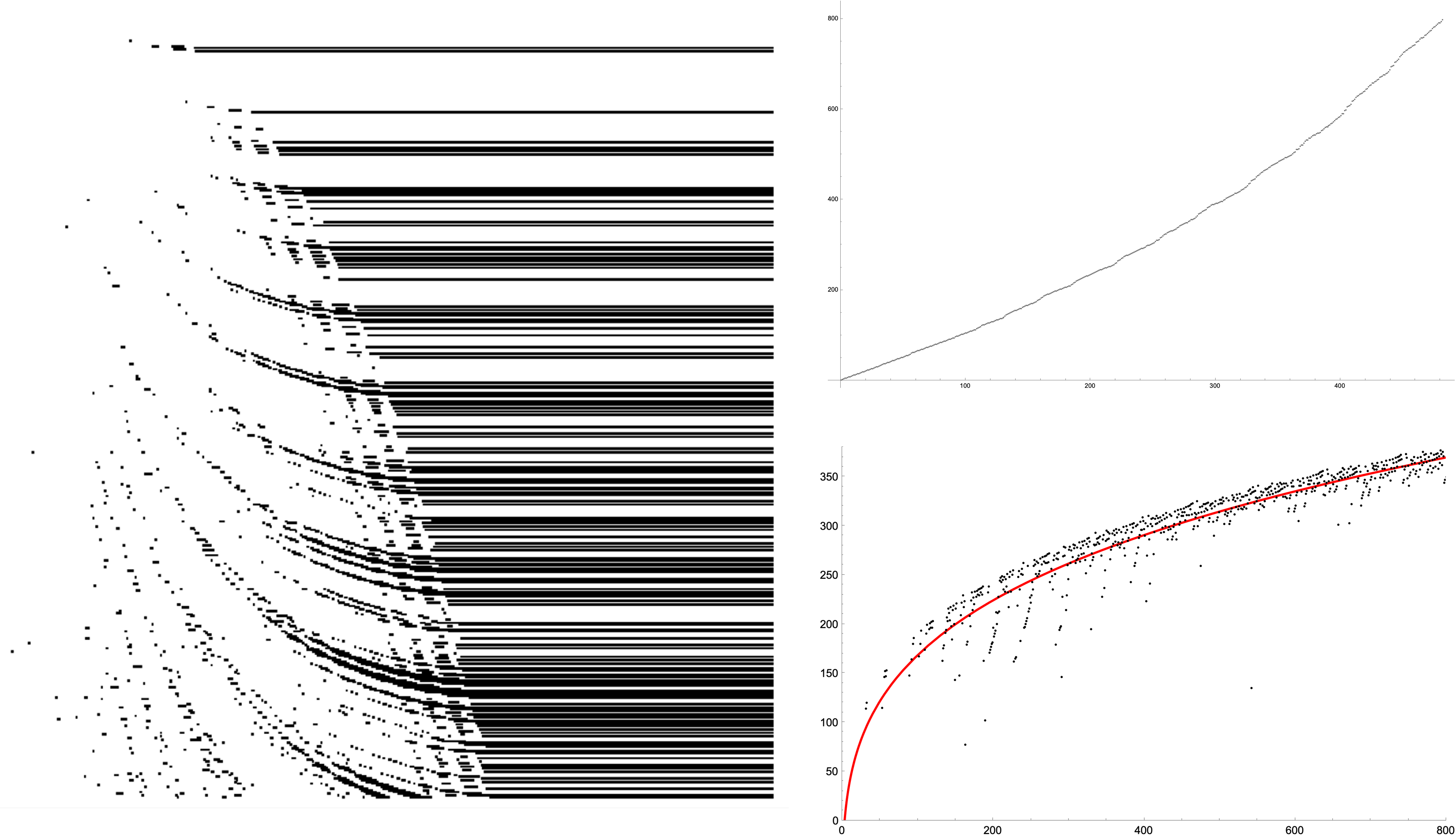}
    \caption{(Left) $\text{Taxicab}(4,j,m)$, (Right Top) $m_i^4$, and (Right Bottom) $J^4(m)$ and $g_4(m)$ for $m \leq 500$ and $7 \leq j \leq 200$}
    \label{bitmap_fourths}
\end{figure}
A familiar pattern emerges in Figure \ref{bitmap_fourths}. There is a chaotic region on the left corresponding to small values of $j$ and a highly regular region on the right corresponding to large values of $j$. These two regions of the graph are again separated by a curve. Let $J^4(m)$ denote the smallest pair $(m,J)$ such that Conjecture \ref{cubes_conj} holds for all $j \geq J$. Figure \ref{bitmap_fourths} (Right Top) graphs $J^4(m)$ for $m \leq 500$. Utilizing Mathematica, a close fit for $J^3(m)$ appears to be $g_4(x) = 43.106\sqrt[4]{x}-28.3045$.

Values of $m$ for which $\text{Taxicab}(4,j,m)$ exists creates the sequence $\{m_i^4\}$ and its complement sequence $\{\mathbb{N} \setminus m_i^4\}$. Figure \ref{bitmap_fourths} (Right Bottom) also plots $J^4(m)$ for $m \leq 500$.
\begin{align*}  
m_i^4&=1,2,\ldots32,34, 35\ldots, 56, 58, 60,61,\ldots, 94,96,97,\ldots\\
\mathbb{N} \setminus m_i^4&=33, 57, 59, 95, 113, 116, 118, 120, 140, 141, 143, 144, 148, 152, 160,\ldots
\end{align*}

For powers higher than $k=4$, there is every reason to suppose that something similar occurs for every $m$ as $j$ gets large. Heuristically, for very large $n$, there are many representations of $n$ as the sum of $j$ positive $k$\textsuperscript{th} powers. To determine if $\text{Taxicab}(k,j,m)$ exists for a particular $m$ and $j$, the only possibilities are relatively small values of $n$. Because there are so few possibilities, there are many values of $\text{Taxicab}(k,j,m)$ that do not exist.

\section{Final Thoughts}
There are seemingly endless variations of partition functions and Taxicab numbers  obtained by manipulating $k$, $j$, and $m$. What is the smallest number that can be expressed as the sum of $j$, \textit{integer} $k$\textsuperscript{th} powers in exactly $m$ ways (without the restriction that the $k$\textsuperscript{th} powers are positive)? For $k=3$ and $j=2$, these numbers are called \textit{Cabtaxi} numbers, denoted $\text{Cabtaxi}(3,j,m)$. Note that $0^k$ is allowed as a summand in this definition. For example, Vi\`ete \cite{Vi`ete} discovered in 1591 that  $\text{Cabtaxi}(3,2,2) = 91 = 3^3+4^3 = 6^2-5^2$. Escott \cite{Escott} computed $\text{Cabtaxi}(3,2,3)$ in 1902, followed by Rathburn \cite{Rathbun} in 1992 who computed $\text{Cabtaxi}(3,2,4),\ldots,\text{Cabtaxi}(3,2,7)$, then Bernstein \cite{Bernstein} in 1998 computed $\text{Cabtaxi}(3,2,8)$, Moore \cite{Moore} in 2005 computed $\text{Cabtaxi}(3,2,9)$, and in 2006, Boyer \cite{Boyer} conjectured correctly $\text{Cabtaxi}(3,2,10)$, which was proven in 2008 by Hollarbach \cite{Hollerbach}. Even with modern computing power, only $10$ Cabtaxi numbers are currently known.
\begin{table}[h!]
    \begin{tabular}{cccc}
        $91=3^3+4^3$&$728= 8^3 + 6^3$&$2741256 = 108^3 + 114^3$&$6017193 = 166^3+113^3$\\
        $\textcolor{white}{91}=6^3-5^3$&$\textcolor{white}{728}= 9^3 - 1^3$&$\textcolor{white}{2741256} = 140^3 - 14^3$&$\textcolor{white}{6017193} = 180^3+57^3$\\
        &$\textcolor{white}{72888}= 12^3 - 10^3$&$\textcolor{white}{2741256} = 168^3 - 126^3$&$\textcolor{white}{6017193} = 185^3-68^3$\\
        &&$\textcolor{white}{2741256} = 207^3 - 183^3$&$\textcolor{white}{6017193} = 209^3-146^3$\\
        &&&$\textcolor{white}{6017193} =246^3-207^3$\\
        &&&\\
        $\text{Cabtaxi}(3,2,2)$&$\text{Cabtaxi}(3,2,3)$&$\text{Cabtaxi}(3,2,4)$&$\text{Cabtaxi}(3,2,5)$
    \end{tabular}
\end{table}
Define \textit{generalized Cabtaxi numbers} as the smallest positive integer $n$ such that $n=x_1^k \pm x_2^k \pm \ldots \pm x_j^k$ in exactly $m$ ways. The relationship with the partition function $p^k(n,j)$ no longer holds because partitions are strictly defined as decreasing sequences of positive integers that \textit{sum} to $n$. Without the restriction to addition, perhaps something analogous occurs with Cabtaxi numbers, or perhaps a completely different pattern will emerge.
\begin{question}
    For each $m$, what is the end behavior of $\text{Cabtaxi}(k,j,m)$ for a fixed $k$ as $j$ gets large?
\end{question}

There are also popular iterations of Taxicab numbers that involve further restrictions on the summands of $n$. What is the smallest number that can be expressed as the sum of $j$ \textit{prime} (or \textit{distinct}, or \textit{coprime}, etc.) positive $k$\textsuperscript{th} powers in exactly $m$ ways? The smallest $\text{Taxicab}(3,2,2)$ and $\text{Cabtaxi}(3,2,2)$ where the summands are prime are:
\begin{align*}
6058655748&=613^3 + 18233^3           &  62540982 &=3973^3 - 313^3\\
&= 10493^3 + 16993^3         &  &= 18673^3 - 18613^3.
\end{align*}

\begin{question}
    For each $m$, what is the end behavior of $\text{Taxicab}(k,j,m)$ and $\text{Cabtaxi}(k,j,m)$ for a fixed $k$ as $j$ gets large and the summands of $n$ are restricted to prime (or distinct, or coprime, etc.) positive $k$\textsuperscript{th} powers?
\end{question}

\section{Competing Interests}
No support was received from any organization for the submitted work nor was any funding received by any author. The authors have no competing interests, financial or otherwise with any organization as relates to this work.

\bibliographystyle{plain}
\typeout{}
\bibliography{taxicab}{}

\begin{thebibliography}{10}

\bibitem{Abramowitz}
M.~Abramowitz, I.~A. Stegun, and R.~H. Romer.
\newblock Handbook of mathematical functions with formulas, graphs, and
  mathematical tables, 1988.

\bibitem{ahlgren}
S.~Ahlgren and J.~Lovejoy.
\newblock The arithmetic of partitions into distinct parts.
\newblock {\em Mathematika}, 48(1-2):203–211, 2001.

\bibitem{Andrews}
G.~E. Andrews.
\newblock {\em The theory of partitions}.
\newblock Number~2. Cambridge University Press, 1998.

\bibitem{Benfield}
B.~Benfield, M.~Paul, and A.~Roy.
\newblock Tur\`an inequalities for $k$-th power partition functions.
\newblock {\em Journal of Mathematical Analysis and Applications},
  529(1):127678, 2024.

\bibitem{Bernt}
B.~C. Berndt.
\newblock {\em Ramanujan’s Notebooks Part IV}.
\newblock Springer New York, NY, 1994.

\bibitem{Bernstein}
D.~J. Bernstein.
\newblock Enumerating solutions to p(a) + q(b) = r(c) + s(d).
\newblock {\em Mathematics of Computation}, 70:389–394, 2001.

\bibitem{Boyer}
C.~Boyer.
\newblock New upper bounds for taxicab and cabtaxi numbers.
\newblock {\em Journal of Integer Sequences}, 11(08.1.6), 2008.

\bibitem{Claude}
C.~S. Calude, E.~Calude, and M.~J. Dinneen.
\newblock What is the value of taxicab(6)?
\newblock {\em Journal of Universal Computer Science}, 9:1196--1203, 2003.

\bibitem{Descartes}
R.~Descartes.
\newblock Oeuvres.
\newblock letters to Mersenne from July 27 and Aug 23, 1638.
\newblock Paris.

\bibitem{Dinitz}
J.~H. Dinitz, R.~Games, and R.~Roth.
\newblock Seeds for generalized taxicab numbers.
\newblock {\em Journal of Integer Sequences}, 22:19.3.3, 2019.

\bibitem{Dubouis}
E.~Dubouis.
\newblock Solution of a problem of {J}. {T}annery.
\newblock {\em Interm\'ediaire Math}, 18:55--56, 1911.

\bibitem{Ekl}
R.~L. Ekl.
\newblock New results in equal sums of like powers.
\newblock {\em Mathematics of Computation}, 67(223):1309–1315, July 1998.

\bibitem{Escott}
E.~B. Escott.
\newblock R\'eponse 1882, trouver quatre nombres tels que la somme de deux
  quel-conques d’entre eux soit un cube.
\newblock {\em L’Interm\'ediaire des Math\'ematiciens}, (9):16–17, 1902.

\bibitem{Euler}
L.~Euler.
\newblock Resolutio formulae {D}iophanteae ab(maa+nbb) = cd(mcc+ndd) per
  numeros rationales.
\newblock {\em Nova Acta Academiae Scientiarum Imperialis Petropolitanae},
  1802.

\bibitem{[Gafni]}
A.~Gafni.
\newblock Power partitions.
\newblock {\em Journal of Number Theory}, 163:19--42, 2016.

\bibitem{Grosswald}
E.~Grosswald.
\newblock {\em Representations of Integers as Sums of Squares}.
\newblock Springer New York, NY, 1985.

\bibitem{Hardy}
G.~H. Hardy.
\newblock {\em Ramanujan}.
\newblock Cambridge Univ. Press, 1940.

\bibitem{Hollerbach}
U.~Hollerbach.
\newblock ``{T}he sixth taxicab number is 24153319581254312065344.", Mar. 8,
  2008.

\bibitem{kim}
D.~Kim and A.~J. Yee.
\newblock A note on partitions into distinct parts and odd parts.
\newblock {\em The Ramanujan Journal}, 3:227--231, 1999.

\bibitem{Lagrange}
J.~L. Lagrange.
\newblock D\'emonstration d'un or\'eme d'arithm\'etique.
\newblock {\em Nouveaux m\'emoires de l'Acad\'emie royale des sciences et
  belles-lettres de Berlin, Anne\'e}, page 123–133, 1770.

\bibitem{Leech}
J.~Leech.
\newblock Some solutions of {D}iophantine equations.
\newblock In {\em Mathematical Proceedings of the Cambridge Philosophical
  Society}, volume~53, pages 778--780. Cambridge University Press, 1957.

\bibitem{Frenicle}
P.~Mancosu.
\newblock Correspondence of {J}ohn {W}allis (1616-1703). vol. {I} (1641-1659),
  2006.

\bibitem{Moore}
D.~Moore.
\newblock Cabtaxi(9).
\newblock NMBRTHRY Archives e-mail, February 5, 2005.

\bibitem{Pall}
G.~Pall.
\newblock On sums of squares.
\newblock {\em The American Mathematical Monthly}, 40(1):10--18, 1933.

\bibitem{Rathbun}
R.~L. Rathbun.
\newblock Sixth taxicab number?
\newblock NMBRTHRY Archives e-mail, July 16, 2002.

\bibitem{Dardis}
E.~Rosenstiel, J.A. Dardis, and C.R. Rosenstiel.
\newblock The four least solutions in distinct positive integers of the
  {D}iophantine equation s= x3+ y3= z3+ w3= u3+ v3= m3+ n3.
\newblock {\em The Institute of Mathematics and its Applications Bulletin},
  27:155--157, 1991.

\bibitem{A025416}
N.~J.~A. Sloane and The OEIS~Foundation Inc.
\newblock The on-line encyclopedia of integer sequences {A}025416.
\newblock http://oeis.org/A025416, 2012.

\bibitem{OEIS}
N.~J.~A. Sloane and The OEIS~Foundation Inc.
\newblock The on-line encyclopedia of integer sequences.
\newblock http://oeis.org/, 2023.

\bibitem{A080673}
N.~J.~A. Sloane and The OEIS~Foundation Inc.
\newblock The on-line encyclopedia of integer sequences {A}080673.
\newblock http://oeis.org/A080673, 2023.

\bibitem{A295702}
N.~J.~A. Sloane and The OEIS~Foundation Inc.
\newblock The on-line encyclopedia of integer sequences {A}295702.
\newblock http://oeis.org/A295702, 2023.

\bibitem{A295795}
N.~J.~A. Sloane and The OEIS~Foundation Inc.
\newblock The on-line encyclopedia of integer sequences {A}295795.
\newblock http://oeis.org/A295795, 2023.

\bibitem{[Vaughan]}
R.~C. Vaughan.
\newblock Squares: additive questions and partitions.
\newblock {\em International Journal of Number Theory}, 11(05):1367--1409,
  2015.

\bibitem{Vi`ete}
D.~T. Whiteside.
\newblock Fran{\c{c}}ois vi{\`e}te, the analytic art. nine studies in algebra,
  geometry and trigonometry from the ‘opus restitutae mathematics analyseos,
  seu algebra nova’.
\newblock {\em The British Journal for the History of Science}, 18(1):98--100,
  1985.

\bibitem{Wilson}
D.~W. Wilson.
\newblock The fifth taxicab number is 48988659276962496.
\newblock {\em J. Integer Seq}, 2(99.1.9), 1999.

\bibitem{[Wright]}
E.~M. Wright.
\newblock Asymptotic partition formulae. {III}. {Partitions} into {{\(k\)}}-th
  powers.
\newblock {\em Acta Mathematica}, 63(1):143--191, 1934.

\end{thebibliography}

\end{document}